\documentclass{scrartcl}

\usepackage{amsmath}
\usepackage{amssymb}
\usepackage{amsthm}
\usepackage{microtype}
\usepackage[hidelinks]{hyperref}
\usepackage{cleveref}
\crefname{equation}{}{} % removes "eq." in front of equation references
\crefname{definition}{Definition}{Definitions}
\crefname{section}{Section}{Sections}
\crefname{theorem}{Theorem}{Theorems}
\crefname{lemma}{Lemma}{Lemmas}
\crefname{assumption}{Assumption}{Assumptions}
\crefname{enumi}{}{}
\usepackage{lmodern}
\usepackage[T1]{fontenc}
\usepackage[backend=biber,giveninits=true,maxbibnames=99]{biblatex}
\usepackage{enumerate}
\usepackage{enumitem}
\usepackage{parskip}

\newcommand{\weakto}{\rightharpoonup}
\newcommand{\Lagrangian}[2]{\mathcal{L} \p{#1, #2}}
\newcommand{\funcFirst}{f}
\newcommand{\funcSecond}{g}
\newcommand{\linOp}{L}
\newcommand\inner[2]{\left\langle #1, #2 \right\rangle}
\newcommand\PrimS{\ensuremath{\mathcal{H}}}
\newcommand\DualS{\ensuremath{\mathcal{G}}}
\newcommand\Real{\ensuremath{\mathbb{R}}}
\newcommand\RealExt{\ensuremath{\Real \cup \set{\infty}}}
\newcommand\p[1]{\ensuremath{\mathord{\left(#1\right)}}}
\newcommand\set[1]{\ensuremath{\mathord{\left\lbrace#1\right\rbrace}}}
\newcommand\setcond[2]{\ensuremath{\mathord{\left\lbrace#1\,\middle\vert\,#2\right\rbrace}}}
\newcommand\norm[1]{\ensuremath{\mathord{\left\lVert#1\right\rVert}}}
\newcommand\inpr[2]{\ensuremath{\mathord{\left\langle#1, #2\right\rangle}}}
\newcommand\Prox[2]{\ensuremath{\mathrm{Prox}^{#1}_{#2}}}

\DeclareMathOperator*{\argmin}{argmin}
\DeclareMathOperator*{\minimize}{minimize}
\DeclareMathOperator*{\maximize}{maximize}

\theoremstyle{plain}
\newtheorem{lemma}{Lemma}
\newtheorem{theorem}{Theorem}
\newtheorem{assumption}{Assumption}
\newtheorem{definition}{Definition}

\title{The Chambolle--Pock method converges weakly with $\theta>1/2$ and $\tau \sigma \|L\|^2<4/(1+2\theta)$}

\author{Sebastian Banert \and Manu Upadhyaya \and Pontus Giselsson}
\date{October 2025}

\begin{document}

\maketitle

\begin{abstract}
  \noindent
  The Chambolle--Pock method is a versatile three-parameter algorithm designed to solve a broad class of composite convex optimization problems, which encompass two proper, lower semicontinuous, and convex functions, along with a linear operator $L$. The functions are accessed via their proximal operators, while the linear operator is evaluated in a forward manner. Among the three algorithm parameters $\tau $, $\sigma $, and $\theta$; $\tau,\sigma >0$ serve as step sizes for the proximal operators, and $\theta$ is an extrapolation step parameter. Previous convergence results have been based on the assumption that $\theta=1$. We demonstrate that weak convergence is achievable whenever $\theta> 1/2$ and $\tau \sigma \|L\|^2<4/\p{1+2\theta}$. Moreover, we establish tightness of the step size bound by providing an example that is nonconvergent whenever the second bound is violated.
\end{abstract}

\section{The Chambolle--Pock method}
The Chambolle--Pock method \cite{chambolleFirstOrderPrimalDualAlgorithm2011}, also known as the primal--dual hybrid gradient method, solves convex-concave saddle-point problems of the form
\begin{align}
  \minimize_{x\in\PrimS}\maximize_{y\in\DualS}\; f\p{x} + \inpr{Lx}{y} - g^*\p{y},
  \label{eq:pd_problem}
\end{align}
where \(f\colon\PrimS\to\RealExt\) and \(g\colon\DualS\to\RealExt\) are proper, convex, and lower semi\-continu\-ous functions, \(g^*\) is the convex conjugate of \(g\), \(L\colon\PrimS\to\DualS\) is a nonzero bounded linear operator, and \(\PrimS\) and \(\DualS\) are real Hilbert spaces. This is a primal--dual formulation of the primal composite optimization problem
\begin{align}
  \minimize_{x\in\PrimS}\; f\p{x} + g \p{L x}.
  \label{eq:prim_problem}
\end{align}
We assume that \eqref{eq:pd_problem} has at least one solution $(x_\star,y_\star)\in\PrimS\times\DualS$ that satisfies the Karush--Kuhn--Tucker (KKT) condition
\begin{subequations}\label{eq:kkt}
    \begin{align}
      - L^* y_\star &\in \partial f \p{x_\star}, \label{eq:kkt:f}\\
      L x_\star &\in \partial g^* \p{y_\star}.\label{eq:kkt:g}
    \end{align}
\end{subequations}
The Chambolle--Pock method searches for such KKT points by iterating
\begin{subequations}\label{eq:CP-iteration}
    \begin{align}
      x_{k+1} &= \Prox{\tau }{f} \p{x_k - \tau  L^* y_k},  \label{eq:CP-iteration:x} \\
      y_{k+1} &= \Prox{\sigma }{g^*} \p{y_k + \sigma  L \p{x_{k+1} + \theta \p{x_{k+1} - x_k}}},\label{eq:CP-iteration:y} 
    \end{align}
\end{subequations}
where each function's proximal operators, the linear operator \(L\) and its adjoint \(L^*\) are evaluated once each in every iteration. 

This algorithm is guaranteed to converge to a solution whenever \(\theta=1\) and the positive step size parameters \(\tau \) and \(\sigma \) satisfy \(\tau \sigma \norm{L}^2 <1\) as shown in the original work \cite[Theorem~1]{chambolleFirstOrderPrimalDualAlgorithm2011}, where \(\norm{L}\) is the operator norm. 
It was soon thereafter realized in \cite{he2012convergence} that for $\theta=1$, the algorithm is an instance of the preconditioned proximal point method, applied to a specific maximally monotone inclusion problem, with a {\emph{strongly positive}} and {\emph{symmetric}} preconditioner.
Strong positivity of the preconditioner is lost for $\theta=1$ whenever \(\tau\sigma\norm{L}^2 \geq 1\), disabling analysis via the standard preconditioned proximal point method. 
In \cite{condat2013primaldual}, sequence convergence with $\tau\sigma\|L\|^2\leq 1$ is established using a modified proximal point analysis and in  \cite{maUnderstandingConvergencePreconditioned2023a,yanImprovedConditionsPrimalDual2024}, convergence to a solution is established for the wider range \(\tau\sigma\norm{L}^2<4/3\). All these results are in finite-dimensional settings and assume $\theta=1$. A modification of the Chambolle--Pock method, which relies on symmetric preconditioning and permits larger values of \(\tau\sigma\|L\|^2\), was presented in \cite{he2022generalized}.

The symmetry of the proximal point preconditioner is lost whenever $\theta\neq 1$. Non-symmetric proximal point iterates are often augmented with different types of correction steps such as projection or momentum correction as in \cite{he2012convergence,latafat2017asymmetric,morinNonlinearForwardBackwardSplitting2023,Giselsson2021Nonlinear} to form convergent algorithms. To the best of our knowledge, no convergence results exist for the pure Chambolle--Pock method, without additional correction, when $\theta\neq 1$.

The contribution of this paper lies in establishing, in infinite dimensional Hilbert spaces, ergodic \(\mathcal{O}\p{1/n}\) primal--dual gap convergence whenever \(\theta\geq 1/2\),
\begin{align}
   \tau \sigma \norm{L}^{2} \leq \frac{4}{1 + 2 \theta}, \label{eq:tau-bounds}
\end{align}
and one of these inequalities holds strictly, as well as weak convergence to a solution whenever both these inequalities are strict. The analysis is based on a Lyapunov inequality that has been derived with the aid of the methodology in \cite{upadhyayaAutomatedTightLyapunov2025}. The type of Lyapunov function that we are using is, to the best of our knowledge, novel and unrelated to the status of the algorithm as a preconditioned proximal point algorithm. We also establish tightness of the step size bound by providing an example of which sequence convergence is lost whenever 
\begin{align}
  \tau  \sigma  \norm{L}^{2} \geq \frac{4}{1 + 2 \theta}.
\end{align}

\section{Preliminaries}
Throughout this paper, \(\PrimS\) and \(\DualS\) denote real Hilbert spaces. All norms $\norm{\cdot}$ are canonical norms where the inner product will be clear from the context. Suppose that \(f\colon \PrimS \to \RealExt\) is a proper, convex, and lower semicontinuous function. The \emph{subdifferential} of \(f\) at a point \(x\in \PrimS\) is defined as
\begin{align}\label{eq:subdifferential-definition}
  \partial f \p{x} = \setcond{y \in \PrimS}{\forall z \in \PrimS, f\p{z} \geq f\p{x} + \inpr{y}{z - x}}.
\end{align}
The \emph{proximal point} of $f$ at \(x \in \PrimS\) with \emph{step size} \(\tau > 0\) is given by
\begin{align*}
  \Prox{\tau}{f} \p{x} = \argmin_{y \in \PrimS} \p{f\p{y} + \frac{1}{2 \tau} \norm{y - x}^2},
\end{align*}
which is uniquely defined. If $x,p\in\PrimS$, we have from \cite[Section~24.1]{Bauschke_Combettes_2017} that
\begin{align}\label{eq:prox-subdifferential}
  p = \Prox{\tau}{f} \p{x} \quad \text{if and only if} \quad \frac{1}{\tau} \p{x - p} \in \partial f \p{p}
\end{align}
and
\begin{align}\label{eq:f_of_prox_finite}
    f \p{\Prox{\tau}{f}(x)}<\infty.
\end{align}
Moreover, the \emph{convex conjugate} of $f$, denoted $f^{*}:\PrimS\to \RealExt$, is the proper, convex, and lower semi\-continu\-ous function given by $f^{*}(u) = \sup_{x\in\PrimS}\p{\inner{u}{x} - f(x)}$ for each $u\in\PrimS$.

\section{Main results}

Our convergence results are derived under the following assumption.
\begin{assumption}
    Assume that
    \begin{enumerate}[label=\roman{enumi})]
    \item \(f \colon \PrimS \to \RealExt\) and \(g \colon \DualS \to \RealExt\) are proper, convex, and lower semicontinuous,
    \item \(L \colon \PrimS \to \DualS\) is a nonzero bounded linear operator, and
    \item there exists at least one point $(x_\star,y_\star)\in\PrimS\times\DualS$ such that both inclusions in \cref{eq:kkt} hold.
    \end{enumerate}
    \label{ass:prob}
\end{assumption}
One of our results concerns a primal--dual gap, which we introduce next. The function given by
\begin{align*}
    \Lagrangian{x}{y} = \funcFirst\p{x} + \inner{y}{\linOp x} - \funcSecond^* \p{y}
\end{align*}
is the \emph{Lagrangian function} associated with \cref{eq:pd_problem}. Given a saddle point \(\p{x_\star, y_\star} \in \PrimS \times \DualS\) satisfying the KKT condition \cref{eq:kkt}, we define the \emph{primal--dual gap function} $\mathcal{D}_{x_\star,y_\star}:\PrimS\times\DualS \to \RealExt$ as
\begin{align}
    \mathcal{D}_{x_\star,y_\star}(x,y) = \Lagrangian{x}{y_\star}-\Lagrangian{x_\star}{y}
    \label{eq:pd_gap_fcn}
\end{align}
for each $x,y\in\PrimS\times\DualS$. It is straightforward to verify that $\mathcal{D}_{x_\star,y_\star}(x,y)\geq 0$ for each $x,y\in\PrimS\times\DualS$ and it will serve as a replacement for function-value suboptimality.
Our first main result shows $\mathcal{O}\p{1/K}$ ergodic convergence for this primal--dual gap.

\begin{theorem}[\textnormal{Ergodic convergence rate}]
    Suppose that \cref{ass:prob} holds. Let $(x_\star,y_\star)\in\PrimS\times\DualS$ be an arbitrary point satisfying \cref{eq:kkt}, let $\p{x_k}_{k=0}^{\infty}$ and $\p{y_k}_{k=0}^{\infty}$ be generated by \cref{eq:CP-iteration} with $\theta\geq 1/2$ and $\tau ,\sigma >0$ satisfying $\tau \sigma \|L\|^2\leq 4/\p{1+2\theta}$ and arbitrary \(\p{x_0, y_0} \in \PrimS \times \DualS\). Further, suppose that at least one of $\theta\geq{1}/{2}$ and $\tau \sigma \|L\|^2\leq 4/\p{1+2\theta}$ holds with strict inequality. Let $\bar{x}_K=\p{\sum_{k=1}^Kx_k}/K$ and $\bar{y}_K=\p{\sum_{k=1}^Ky_k}/K$. Then for each positive integer $K$, the ergodic duality gap $\mathcal{D}_{x_\star,y_\star}(\bar{x}_K,\bar{y}_K)$ converges as $\mathcal{O}\p{1/K}$.
  \label{thm:pd}
\end{theorem}

Our second main result shows weak sequence convergence.

\begin{theorem}[\textnormal{Weak convergence}]
    Suppose that \cref{ass:prob} holds. Let $\p{x_k}_{k=0}^{\infty}$ and $\p{y_k}_{k=0}^{\infty}$ be generated by \cref{eq:CP-iteration} with $\theta > 1/2$ and $\tau ,\sigma >0$ satisfying $\tau \sigma \|L\|^2<4/\p{1+2\theta}$ and arbitrary \(\p{x_0, y_0} \in \PrimS \times \DualS\).  Then $(x_{k},y_k)\weakto(x_\star,y_\star)$ for some KKT point $(x_\star,y_\star)\in\PrimS\times\DualS$ satisfying \cref{eq:kkt}.
  \label{thm:sequence}
\end{theorem} 

Note that \cref{thm:pd} allows for equality in \(\tau  \sigma  \norm{L}^{2} \leq 4/\p{1 + 2 \theta}\), while \cref{thm:sequence} requires strict inequality. 

The remainder of the paper is devoted to proving these results and giving a tightness guarantee for \cref{thm:sequence}. In \cref{sec:lemma}, we introduce two lemmas, which subsequently find their application in \cref{sec:Lyapunov}, where our Lyapunov analysis is presented. \Cref{sec:thm_proofs} proves \cref{thm:pd,thm:sequence} while \cref{sec:counterexample} provides an example for which the sequence fails to converge whenever $\tau \sigma \|L\|^2\geq 4/(1+2\theta)$, implying that \cref{thm:sequence} is tight in this sense. Finally, \cref{sec:boundary} shows that $\|x_{k+2}-x_k\|\to 0$ on the boundary of the step-size condition, i.e., whenever $\tau \sigma \|L\|^2=4/(1+2\theta)$.

\section{Two lemmas}\label{sec:lemma}

Our Lyapunov analysis entails evaluating the primal--dual gap function $\mathcal{D}_{x_\star,y_\star}$ at points generated by the algorithm. To this end, we introduce the following definition.

\begin{definition}\label{def:FkGk}
    Suppose that \cref{ass:prob} holds. Let the sequences \(\p{x_k}_{k =0}^{\infty}\) and \(\p{y_k}_{k =0}^{\infty}\) be generated by \cref{eq:CP-iteration} for some initial point $\p{x_0,y_0}\in \PrimS\times \DualS$, and let \(\p{x_\star, y_\star} \in \PrimS\times \DualS\) satisfy \cref{eq:kkt}. Then we define the sequences $\p{F_k}_{k=0}^{\infty}$ and $\p{G_k}_{k=0}^{\infty}$ by
  \begin{align*}
    F_k = f\p{x_{k+1}} - f\p{x_\star} + \inpr{L^* y_\star}{x_{k+1} - x_\star}
  \end{align*}
  and
  \begin{align*}
    G_k = g^*\p{y_{k+1}} - g^* \p{y_\star} - \inpr{L x_\star}{y_{k+1} - y_\star},
  \end{align*}
  respectively. 
\end{definition}

Next, we present two lemmas concerning $F_k$ and $G_k$.

\begin{lemma}\label{lem:FkGk_properties}
    Suppose that $\p{F_k}_{k=0}^{\infty}$ and $\p{G_k}_{k=0}^{\infty}$ are given as in Definition~\ref{def:FkGk}. Then
  \begin{enumerate}[label={\roman{enumi})}]
    \item\label{lem3:1} $F_k,G_k\geq 0$,
    \item\label{lem3:2} $F_k+G_k<\infty$, and
    \item\label{lem3:3} $F_k+G_k=\mathcal{D}_{x_\star,y_\star}(x_{k+1},y_{k+1})$.
  \end{enumerate}
\end{lemma}
\begin{proof}
\textit{\cref{lem3:1}} By \cref{eq:kkt:f,eq:subdifferential-definition}, we get that
  \begin{align*}
    f\p{x_{k+1}} \geq f\p{x_\star} - \inpr{L^* y_\star}{x_{k+1} - x_\star},
  \end{align*}
  which implies that $F_k$ is nonnegative.
  By \cref{eq:kkt:g,eq:subdifferential-definition}, we get that
  \begin{align*}
    g^*\p{y_{k+1}} \geq g^* \p{y_\star} + \inpr{L x_\star}{y_{k+1} - y_\star},
  \end{align*}
  which implies that $G_k$ is nonnegative.
  
  \textit{\cref{lem3:2}} By \cref{eq:f_of_prox_finite} and \cref{eq:kkt}, we conclude that all terms defining $F_k$ and $G_k$ are finite. Hence $F_k+G_k<\infty$.
  
  \textit{\cref{lem3:3}} We have
  \begin{align*}
    F_k+G_k
    &= \funcFirst\p{x_{k+1}} + \inner{y_\star}{\linOp x_{k+1}} - \funcSecond^* \p{y_\star} - (\funcFirst\p{x_\star} + \inner{y_{k+1}}{\linOp x_\star} - \funcSecond^* \p{y_{k+1}}   )\\&= \Lagrangian{x_{k+1}}{y_\star} - \Lagrangian{x_\star}{y_{k+1}}.\qedhere
\end{align*}
\end{proof}
We call \(F_k + G_k\) a \emph{primal--dual gap} at iteration \(k\). Note that the primal--dual gap depends on the particular choice of KKT point $(x_\star,y_\star)$ and our analysis is valid for each such choice. 

Before introducing the next lemma related to $F_k$ and $G_k$, we conclude, using \cref{eq:prox-subdifferential}, that
the iterative procedure in \cref{eq:CP-iteration} is equivalent to the following set of inclusions:
\begin{align}
  \frac{1}{\tau } \p{x_k - x_{k+1}} - L^* y_k &\in \partial f \p{x_{k+1}}, \label{eq:CP-inclusion:f} \\
  \frac{1}{\sigma } \p{y_k - y_{k+1}} + L \p{x_{k+1} + \theta \p{x_{k+1} - x_k}} &\in \partial g^* \p{y_{k+1}}. \label{eq:CP-inclusion:g}
\end{align}

\begin{lemma}\label{lem:interpolations}
  Suppose that $\p{F_k}_{k=0}^{\infty}$ and $\p{G_k}_{k=0}^{\infty}$ are given as in Definition~\ref{def:FkGk}. Then
  \begin{align*}
    G_k &\leq \frac{1}{\sigma } \inpr{y_k - y_{k+1}}{y_{k+1} - y_\star} + \inpr{L x_{k+1} - L x_\star}{y_{k+1} - y_\star} \\&\qquad + \theta \inpr{L x_{k+1} - L x_k}{y_{k+1} - y_\star}, \\
    F_{k+1} &\leq \frac{1}{\tau } \inpr{x_{k+1} - x_{k+2}}{x_{k+2} - x_\star} + \inpr{y_\star - y_{k+1}}{L x_{k+2} - L x_\star}, \\
    F_{k+1} - F_k &\leq  - \frac{1}{\tau } \norm{x_{k+2} - x_{k+1}}^2 + \inpr{y_\star - y_{k+1}}{L x_{k+2} - L x_{k+1}}, \\
    F_k - F_{k+1} &\leq \frac{1}{\tau } \inpr{x_k - x_{k+1}}{x_{k+1} - x_{k+2}} + \inpr{y_\star - y_k}{L x_{k+1} - L x_{k+2}}.
  \end{align*}
\end{lemma}
\begin{proof}
  By \cref{eq:CP-inclusion:g,eq:subdifferential-definition},
  \begin{align*}
    G_k
    &= g^*\p{y_{k+1}} - g^* \p{y_\star} - \inpr{L x_\star}{y_{k+1} - y_\star} \\
    &\leq \inpr{\frac{1}{\sigma } \p{y_k - y_{k+1}} + L \p{x_{k+1} + \theta \p{x_{k+1} - x_k}}}{y_{k+1} - y_\star} - \inpr{L x_\star}{y_{k+1} - y_\star} \\
    &= \frac{1}{\sigma } \inpr{y_k - y_{k+1}}{y_{k+1} - y_\star} + \inpr{L x_{k+1} - L x_\star}{y_{k+1} - y_\star} + \theta \inpr{L x_{k+1} - L x_k}{y_{k+1} - y_\star}.
  \end{align*}
  By \cref{eq:CP-inclusion:f,eq:subdifferential-definition},
  \begin{align*}
    F_{k+1}
    &= f\p{x_{k+2}} - f\p{x_\star} + \inpr{L^* y_\star}{x_{k+2} - x_\star} \\
    &\leq \inpr{\frac{1}{\tau } \p{x_{k+1} - x_{k+2}} - L^* y_{k+1}}{x_{k+2} - x_\star} + \inpr{L^* y_\star}{x_{k+2} - x_\star} \\
    &= \frac{1}{\tau } \inpr{x_{k+1} - x_{k+2}}{x_{k+2} - x_\star} + \inpr{y_\star - y_{k+1}}{L x_{k+2} - L x_\star}
  \end{align*}
  and
  \begin{align*}
    F_{k+1} - F_k
    &= f\p{x_{k+2}} - f\p{x_{k+1}} + \inpr{L^* y_\star}{x_{k+2} - x_{k+1}} \\
    &\leq \inpr{\frac{1}{\tau } \p{x_{k+1} - x_{k+2}} - L^* y_{k+1}}{x_{k+2} - x_{k+1}} + \inpr{L^* y_\star}{x_{k+2} - x_{k+1}} \\
    &= - \frac{1}{\tau } \norm{x_{k+2} - x_{k+1}}^2 + \inpr{y_\star - y_{k+1}}{L x_{k+2} - L x_{k+1}}
  \end{align*}
  and
  \begin{align*}
    F_k - F_{k+1}
    &= f\p{x_{k+1}} - f\p{x_{k+2}} + \inpr{L^* y_\star}{x_{k+1} - x_{k+2}} \\
    &\leq \frac{1}{\tau } \inpr{x_k - x_{k+1}}{x_{k+1} - x_{k+2}} + \inpr{y_\star - y_k}{L x_{k+1} - L x_{k+2}}. \qedhere
  \end{align*}
\end{proof}

\section{Lyapunov analysis}
\label{sec:Lyapunov}

In this section, we provide a Lyapunov analysis of the Chambolle–Pock method, which serves as the basis for our convergence results. Before presenting this, we establish the nonnegativity of a certain coefficient that emerges in the Lyapunov analysis.
\begin{lemma}
    \label{lem:coeff_pos}
    Suppose that $\theta\geq {1}/{2}$ and $\tau ,\sigma >0$, and that they jointly satisfy $0<\tau \sigma \|L\|^2\leq{4}/{\p{1+2\theta}}$. Further suppose that at least one of $\theta\geq{1}/{2}$ and $\tau \sigma \|L\|^2\leq{4}/{\p{1+2\theta}}$ holds with strict inequality. Then $8 \theta - \tau  \sigma  \norm{L}^2 \p{4 \theta^2 + 1}$ is positive.
\end{lemma}
\begin{proof}
    We have
    \begin{align*}
        8 \theta - \tau  \sigma  \norm{L}^2 \p{4 \theta^2 + 1}  &\geq 8 \theta -  \frac{4\p{4 \theta^2 + 1}}{1+2\theta} =\frac{8 \theta - 4}{1+2\theta} =4-\frac{8}{1+2\theta} \geq 0,
    \end{align*}
    where the first inequality is strict if $\tau \sigma \|L\|^2<{4}/{\p{1+2\theta}}$ and the last inequality is strict if $\theta>{1}/{2}$.
\end{proof}

\begin{lemma}[Lyapunov inequality]\label{lem:lyapunov-inequality}
  Suppose that \cref{ass:prob} holds. Let $(x_\star,y_\star)\in\PrimS\times\DualS$ be an arbitrary point satisfying \cref{eq:kkt}, let $\p{x_k}_{k=0}^{\infty}$ and $\p{y_k}_{k=0}^{\infty}$ be generated by \cref{eq:CP-iteration} with $\theta\geq {1}/{2}$ and $\tau ,\sigma >0$ satisfying $\tau \sigma \|L\|^2\leq{4}/{\p{1+2\theta}}$, and arbitrary \(\p{x_0, y_0} \in \PrimS \times \DualS\). Further, suppose that at least one of $\theta\geq{1}/{2}$ and $\tau \sigma \|L\|^2\leq{4}/{\p{1+2\theta}}$ holds with strict inequality. Moreover, for each nonnegative integer \(k\), let
  \begin{multline*}
    V_k = \theta F_k + \frac{1}{2 \tau } \norm{x_{k+1} - x_\star}^2 + \frac{1}{2 \sigma } \norm{y_k - y_\star + \sigma  \theta \p{L x_{k+1} - L x_k}}^2 \\
    + \frac{\theta}{2 \tau } \norm{x_{k+1} - x_k}^2 - \frac{\sigma  \p{4 \theta^2 + 1}}{16} \norm{L x_{k+1} - L x_k}^2
  \end{multline*}
  and $F_k$ and $G_k$ be as in \cref{def:FkGk}.
  Then, for each nonnegative integer $k$, $V_k\geq 0$ and
  \begin{align*}
    V_{k+1} - &\,V_k + F_k + G_k \leq  - \frac{1}{2 \sigma } \norm{y_{k+1} - y_k - \sigma  \p{\frac{1}{2} \p{L x_{k+1} - L x_{k+2}} - \theta \p{L x_k - L x_{k+1}}}}^2 \\
     &- \frac{8 \theta - \tau  \sigma  \norm{L}^2 \p{4 \theta^2 + 1}}{16 \tau } \norm{x_{k+2} - x_{k+1} - \frac{4 \p{1 - \tau \sigma  \theta \norm{L}^2}}{8 \theta - \tau  \sigma  \norm{L}^2 \p{4 \theta^2 + 1}} \p{x_{k+1} - x_k}}^2\\
    &- \frac{\p{4 \theta^2 - 1} \p{4 - \tau  \sigma  \norm{L}^2 \p{2 \theta + 1}} \p{4 - \tau  \sigma  \norm{L}^2 \p{2 \theta - 1}}}{16 \tau  \p{8 \theta - \tau  \sigma  \norm{L}^2 \p{4 \theta^2 + 1}}} \norm{x_{k+1} - x_k}^2,
  \end{align*}
  where the latter inequality is called a \emph{Lyapunov inequality}.
\end{lemma}
\begin{proof}
We start by proving nonnegativity of $V_k$. Since $\theta\geq {1}/{2}$, $F_k\geq 0$ by \cref{lem:FkGk_properties}, and $\tau ,\sigma >0$, we conclude that $V_k\geq 0$ since
\begin{multline*}
  \frac{\theta}{2 \tau } \norm{x_{k+1} - x_k}^2 - \frac{\sigma  \p{4 \theta^2 + 1}}{16} \norm{L x_{k+1} - L x_k}^2 \\
  \begin{aligned}
    &\geq \frac{\theta}{2 \tau } \norm{x_{k+1} - x_k}^2 - \frac{\sigma  \|L\|^2\p{4 \theta^2 + 1}}{16} \norm{x_{k+1} - x_k}^2 \\
    &= \frac{1}{\tau }\p{\frac{8\theta- \tau \sigma  \|L\|^2\p{4 \theta^2 + 1}}{16}}\norm{x_{k+1} - x_k}^2 \geq 0,
  \end{aligned}
\end{multline*}
  where the last inequality follows from \cref{lem:coeff_pos}.

We next prove the Lyapunov inequality.
  By \cref{lem:interpolations}, we have
  \begin{align*}
    &\theta F_{k+1} - \theta F_k + F_k + G_k \\  
    &\quad = \p{\theta - \frac{1}{2}} \p{F_{k+1} - F_k} + F_{k+1} + \frac{1}{2} \p{F_k - F_{k+1}} + G_k \\
    &\quad 
    \begin{aligned}
        &\leq \p{\theta - \frac{1}{2}} \p{ - \frac{1}{\tau } \norm{x_{k+2} - x_{k+1}}^2 + \inpr{y_\star - y_{k+1}}{L x_{k+2} - L x_{k+1}}} \\
        &\quad + \frac{1}{\tau } \inpr{x_{k+1} - x_{k+2}}{x_{k+2} - x_\star} + \inpr{y_\star - y_{k+1}}{L x_{k+2} - L x_\star} \\
        &\quad + \frac{1}{2} \p{\frac{1}{\tau } \inpr{x_k - x_{k+1}}{x_{k+1} - x_{k+2}} + \inpr{y_\star - y_k}{L x_{k+1} - L x_{k+2}}} \\
        &\quad + \frac{1}{\sigma } \inpr{y_k - y_{k+1}}{y_{k+1} - y_\star} + \inpr{L x_{k+1} - L x_\star}{y_{k+1} - y_\star} \\
        &\quad + \theta \inpr{L x_{k+1} - L x_k}{y_{k+1} - y_\star} 
    \end{aligned} \\
    &\quad 
    \begin{aligned}
      &=- \frac{\theta}{\tau } \norm{x_{k+2} - x_{k+1}}^2 + \theta \inpr{y_\star - y_{k+1}}{L x_{k+2} - 2 L x_{k+1} + L x_k} \\
      &\quad + \frac{1}{2 \tau } \norm{x_{k+2} - x_{k+1}}^2 + \frac{1}{2 \tau } \inpr{x_k - x_{k+1}}{x_{k+1} - x_{k+2}} \\
      &\quad + \frac{1}{\tau } \inpr{x_{k+1} - x_{k+2}}{x_{k+2} - x_\star} + \inpr{y_\star - y_{k+1}}{L x_{k+2} - L x_{k+1}} \\
      &\quad + \frac{1}{2} \inpr{2 y_\star - y_k - y_{k+1}}{L x_{k+1} - L x_{k+2}} \\
      &\quad + \frac{1}{\sigma } \inpr{y_k - y_{k+1}}{y_{k+1} - y_\star}.
    \end{aligned}
  \end{align*}
  Moreover,
  \begin{multline*}
    \frac{1}{2 \tau } \norm{x_{k+2} - x_\star}^2 - \frac{1}{2 \tau } \norm{x_{k+1} - x_\star}^2 \\
    = \frac{1}{2 \tau } \norm{x_{k+2} - x_{k+1}}^2 + \frac{1}{\tau } \inpr{x_{k+2} - x_{k+1}}{x_{k+1} - x_\star}
  \end{multline*}
  and
  \begin{multline*}
    \frac{1}{2 \sigma } \norm{y_{k+1} - y_\star + \sigma  \theta \p{L x_{k+2} - L x_{k+1}}}^2 - \frac{1}{2 \sigma } \norm{y_k - y_\star + \sigma  \theta \p{L x_{k+1} - L x_k}}^2 \\
    = \frac{1}{2 \sigma } \norm{y_{k+1} - y_k}^2 + \frac{1}{\sigma } \inpr{y_{k+1} - y_k}{y_k - y_\star} \\
    + \frac{\sigma  \theta^2}{2} \norm{L x_{k+2} - L x_{k+1}}^2 - \frac{\sigma  \theta^2}{2} \norm{L x_{k+1} - L x_k}^2 \\
    + \theta \inpr{y_{k+1} - y_\star}{L x_{k+2} - L x_{k+1}} - \theta \inpr{y_k - y_\star}{L x_{k+1} - L x_k}.
  \end{multline*}
  Hence,
  \begin{align*}
      & V_{k+1} - V_k + F_k + G_k \\
      & \quad 
      \begin{aligned}
        &= \theta F_{k+1} - \theta F_k + F_k + G_k + \frac{1}{2 \tau } \norm{x_{k+2} - x_\star}^2 - \frac{1}{2 \tau } \norm{x_{k+1} - x_\star}^2 \\
        &\quad + \frac{1}{2 \sigma } \norm{y_{k+1} - y_\star + \sigma  \theta \p{L x_{k+2} - L x_{k+1}}}^2 - \frac{1}{2 \sigma } \norm{y_k - y_\star + \sigma  \theta \p{L x_{k+1} - L x_k}}^2 \\
        &\quad + \frac{\theta}{2 \tau } \norm{x_{k+2} - x_{k+1}}^2 - \frac{\sigma  \p{4 \theta^2 + 1}}{16} \norm{L x_{k+2} - L x_{k+1}}^2 \\
        &\quad - \frac{\theta}{2 \tau } \norm{x_{k+1} - x_k}^2 + \frac{\sigma  \p{4 \theta^2 + 1}}{16} \norm{L x_{k+1} - L x_k}^2 
      \end{aligned} \\
      & \quad 
      \begin{aligned}
          &\leq  - \frac{\theta}{\tau } \norm{x_{k+2} - x_{k+1}}^2 + \theta \inpr{y_\star - y_{k+1}}{L x_{k+2} - 2 L x_{k+1} + L x_k} \\
          &\qquad + \frac{1}{2 \tau } \norm{x_{k+2} - x_{k+1}}^2 + \frac{1}{2 \tau } \inpr{x_k - x_{k+1}}{x_{k+1} - x_{k+2}} \\
          &\qquad + \frac{1}{\tau } \inpr{x_{k+1} - x_{k+2}}{x_{k+2} - x_\star} + \inpr{y_\star - y_{k+1}}{L x_{k+2} - L x_{k+1}} \\
          &\quad + \frac{1}{2} \inpr{2 y_\star - y_k - y_{k+1}}{L x_{k+1} - L x_{k+2}} \\
          &\quad + \frac{1}{\sigma } \inpr{y_k - y_{k+1}}{y_{k+1} - y_\star} \\
          &\quad + \frac{1}{2 \tau } \norm{x_{k+2} - x_{k+1}}^2 + \frac{1}{\tau } \inpr{x_{k+2} - x_{k+1}}{x_{k+1} - x_\star} \\
          &\qquad + \frac{1}{2 \sigma } \norm{y_{k+1} - y_k}^2 + \frac{1}{\sigma } \inpr{y_{k+1} - y_k}{y_k - y_\star} \\
          &\quad + \frac{\sigma  \theta^2}{2} \norm{L x_{k+2} - L x_{k+1}}^2 - \frac{\sigma  \theta^2}{2} \norm{L x_{k+1} - L x_k}^2 \\
          &\qquad + \theta \inpr{y_{k+1} - y_\star}{L x_{k+2} - L x_{k+1}} - \theta \inpr{y_k - y_\star}{L x_{k+1} - L x_k} \\
          &\quad + \frac{\theta}{2 \tau } \norm{x_{k+2} - x_{k+1}}^2 - \frac{\sigma  \p{4 \theta^2 + 1}}{16} \norm{L x_{k+2} - L x_{k+1}}^2 \\
          &\quad - \frac{\theta}{2 \tau } \norm{x_{k+1} - x_k}^2 + \frac{\sigma  \p{4 \theta^2 + 1}}{16} \norm{L x_{k+1} - L x_k}^2 \\
      \end{aligned} \\
      & \quad 
      \begin{aligned}
            &=  - \frac{1}{2 \sigma } \norm{y_{k+1} - y_k}^2 + \inpr{y_{k+1} - y_k}{\frac{1}{2} \p{L x_{k+1} - L x_{k+2}} - \theta \p{L x_k - L x_{k+1}}} \\
            &\quad - \frac{\theta}{2 \tau } \norm{x_{k+2} - x_{k+1}}^2 + \frac{1}{2 \tau } \inpr{x_k - x_{k+1}}{x_{k+1} - x_{k+2}} - \frac{\theta}{2 \tau } \norm{x_{k+1} - x_k}^2 \\
            & \quad + \frac{\sigma  \p{4 \theta^2 - 1}}{16} \norm{L x_{k+2} - L x_{k+1}}^2 - \frac{\sigma  \p{4 \theta^2 - 1}}{16} \norm{L x_{k+1} - L x_k}^2.
      \end{aligned}
  \end{align*}
  Completing the square yields
  \begin{align*}
    & V_{k+1} - V_k + F_k + G_k \\
    & \quad 
    \begin{aligned}
        &\leq  -\frac{1}{2 \sigma } \norm{y_{k+1} - y_k - \sigma  \p{\frac{1}{2} \p{L x_{k+1} - L x_{k+2}} - \theta \p{L x_k - L x_{k+1}}}}^2 \\
        &\quad + \frac{\sigma }{2} \norm{\frac{1}{2} \p{L x_{k+1} - L x_{k+2}} - \theta \p{L x_k - L x_{k+1}}}^2 \\
        &\quad - \frac{\theta}{2 \tau } \norm{x_{k+2} - x_{k+1}}^2 + \frac{1}{2 \tau } \inpr{x_k - x_{k+1}}{x_{k+1} - x_{k+2}} - \frac{\theta}{2 \tau } \norm{x_{k+1} - x_k}^2 \\
        &\quad + \frac{\sigma  \p{4 \theta^2 - 1}}{16} \norm{L x_{k+2} - L x_{k+1}}^2 - \frac{\sigma  \p{4 \theta^2 - 1}}{16} \norm{L x_{k+1} - L x_k}^2 
    \end{aligned} \\
    & \quad 
    \begin{aligned}
        &= - \frac{1}{2 \sigma } \norm{y_{k+1} - y_k - \sigma  \p{\frac{1}{2} \p{L x_{k+1} - L x_{k+2}} - \theta \p{L x_k - L x_{k+1}}}}^2 \\
        &\quad - \frac{\sigma  \theta}{2} \inpr{L x_{k+1} - L x_{k+2}}{L x_k - L x_{k+1}} \\
        &\quad - \frac{\theta}{2 \tau } \norm{x_{k+2} - x_{k+1}}^2 + \frac{1}{2 \tau } \inpr{x_k - x_{k+1}}{x_{k+1} - x_{k+2}} - \frac{\theta}{2 \tau } \norm{x_{k+1} - x_k}^2 \\
        &\quad + \frac{\sigma  \p{4 \theta^2 + 1}}{16} \norm{L x_{k+2} - L x_{k+1}}^2 + \frac{\sigma  \p{4 \theta^2 + 1}}{16} \norm{L x_{k+1} - L x_k}^2
    \end{aligned} \\
    & \quad 
    \begin{aligned}
        &=- \frac{1}{2 \sigma } \norm{y_{k+1} - y_k - \sigma  \p{\frac{1}{2} \p{L x_{k+1} - L x_{k+2}} - \theta \p{L x_k - L x_{k+1}}}}^2 \\
        &\quad + \frac{\sigma  \p{4 \theta^2 + 1}}{16} \norm{L x_{k+2} - L x_{k+1} + \frac{4 \theta}{4 \theta^2 + 1} \p{L x_k - L x_{k+1}}}^2 \\
        &\quad - \frac{\theta}{2 \tau } \norm{x_{k+2} - x_{k+1}}^2 + \frac{1}{2 \tau } \inpr{x_k - x_{k+1}}{x_{k+1} - x_{k+2}} - \frac{\theta}{2 \tau } \norm{x_{k+1} - x_k}^2 \\
        &\quad + \frac{\sigma  \p{4 \theta^2 - 1}^2}{16 \p{4 \theta^2 + 1}} \norm{L x_{k+1} - L x_k}^2.
    \end{aligned} 
  \end{align*}
  Using the Lipschitz continuity of \(L\) and that $\theta\geq 1/2$ give
  \begin{align*}
    & V_{k+1} - V_k + F_k + G_k \\
    &\quad 
    \begin{aligned}
        &\leq  - \frac{1}{2 \sigma } \norm{y_{k+1} - y_k - \sigma  \p{\frac{1}{2} \p{L x_{k+1} - L x_{k+2}} - \theta \p{L x_k - L x_{k+1}}}}^2 \\
        &\quad + \frac{\sigma  \norm{L}^2 \p{4 \theta^2 + 1}}{16} \norm{x_{k+2} - x_{k+1} + \frac{4 \theta}{4 \theta^2 + 1} \p{x_k - x_{k+1}}}^2 \\
        &\quad - \frac{\theta}{2 \tau } \norm{x_{k+2} - x_{k+1}}^2 + \frac{1}{2 \tau } \inpr{x_k - x_{k+1}}{x_{k+1} - x_{k+2}} - \frac{\theta}{2 \tau } \norm{x_{k+1} - x_k}^2 \\
        &\quad + \frac{\sigma  \norm{L}^2 \p{4 \theta^2 - 1}^2}{16 \p{4 \theta^2 + 1}} \norm{x_{k+1} - x_k}^2
    \end{aligned} \\
    &\quad 
    \begin{aligned}
        &= - \frac{1}{2 \sigma } \norm{y_{k+1} - y_k - \sigma  \p{\frac{1}{2} \p{L x_{k+1} - L x_{k+2}} - \theta \p{L x_k - L x_{k+1}}}}^2 \\
        &\quad - \frac{8 \theta - \tau  \sigma  \norm{L}^2 \p{4 \theta^2 + 1}}{16 \tau } \norm{x_{k+2} - x_{k+1}}^2 \\
        &\quad + \frac{1 - \tau \sigma  \theta \norm{L}^2}{2 \tau } \inpr{x_{k+2} - x_{k+1}}{x_{k+1} - x_k} \\
        &\quad - \frac{8 \theta - \tau  \sigma  \norm{L}^2 \p{4 \theta^2 + 1}}{16 \tau } \norm{x_{k+1} - x_k}^2 
    \end{aligned} \\
    &\quad 
    \begin{aligned}
        &=- \frac{1}{2 \sigma } \norm{y_{k+1} - y_k - \sigma  \p{\frac{1}{2} \p{L x_{k+1} - L x_{k+2}} - \theta \p{L x_k - L x_{k+1}}}}^2 \\
        &\quad - \frac{8 \theta - \tau  \sigma  \norm{L}^2 \p{4 \theta^2 + 1}}{16 \tau } \norm{x_{k+2} - x_{k+1} - \frac{4 \p{1 - \tau \sigma  \theta \norm{L}^2}}{8 \theta - \tau  \sigma  \norm{L}^2 \p{4 \theta^2 + 1}} \p{x_{k+1} - x_k}}^2 \\
        &\quad - \frac{\p{4 \theta^2 - 1} \p{4 - \tau  \sigma  \norm{L}^2 \p{2 \theta + 1}} \p{4 - \tau  \sigma  \norm{L}^2 \p{2 \theta - 1}}}{16 \tau  \p{8 \theta - \tau  \sigma  \norm{L}^2 \p{4 \theta^2 + 1}}} \norm{x_{k+1} - x_k}^2,
    \end{aligned} 
  \end{align*}
   where the last completion of squares equality holds since $\p{8 \theta - \tau  \sigma  \norm{L}^2 \p{4 \theta^2 + 1}}>0$ by \cref{lem:coeff_pos}.\qedhere
\end{proof}

\section{Theorem proofs}
\label{sec:thm_proofs}

In this section, we use \cref{lem:lyapunov-inequality} to prove \cref{thm:pd,thm:sequence}, i.e., we prove ergodic convergence of the primal--dual gap and weak convergence of the sequences \(\p{x_k}_{k=0}^{\infty}\) and \(\p{y_k}_{k=0}^{\infty}\) to a KKT point. Given the Lyapunov inequality, the arguments are standard. However, before we prove these results, we state the following lemma on nonnegativity of a coefficient in the Lyapunov inequality.
\begin{lemma}
    \label{lem:coeff_nonneg}
    Suppose that $\theta\geq {1}/{2}$ and $\tau ,\sigma >0$ and that they jointly satisfy $0<\tau \sigma \|L\|^2\leq{4}/{\p{1+2\theta}}$. Further suppose that at least one of $\theta\geq{1}/{2}$ and $\tau \sigma \|L\|^2\leq 4/\p{1+2\theta}$ holds with strict inequality. Then 
    \begin{align*}
    \frac{\p{4 \theta^2 - 1} \p{4 - \tau  \sigma  \norm{L}^2 \p{2 \theta + 1}} \p{4 - \tau  \sigma  \norm{L}^2 \p{2 \theta - 1}}}{16 \tau  \p{8 \theta - \tau  \sigma  \norm{L}^2 \p{4 \theta^2 + 1}}}\geq 0
    \end{align*}
    with strict inequality if $\tau \sigma \|L\|^2<{4}/{\p{1+2\theta}}$ and $\theta>{1}/{2}$ and equality if $\tau \sigma \|L\|^2={4}/{\p{1+2\theta}}$ or $\theta=1/2$.
\end{lemma}
\begin{proof}
We know from \cref{lem:coeff_pos} that the denominator is positive. The numerator consists of three factors. The first factor is positive if $\theta>{1}/{2}$ and 0 if $\theta={1}/{2}$. The third factor is larger than the second. The second factor is positive if $\tau \sigma \|L\|^2<{4}/{\p{1+2\theta}}$ and 0 if $\tau \sigma \|L\|^2={4}/{\p{1+2\theta}}$. Combining these gives the result.
\end{proof}

\subsection{Proof of \texorpdfstring{\cref{thm:pd}}{Theorem 1} }
\begin{proof}
Due to \cref{lem:lyapunov-inequality,lem:coeff_pos,lem:FkGk_properties,lem:coeff_nonneg}, we conclude, using a telescoping summation, that $F_k+G_k$ is summable and satisfies
\begin{align*}
    0\leq\sum_{k=0}^{K-1}\p{F_k+G_k}\leq V_0<\infty
\end{align*}
for each positive integer $K$. From Jensen's inequality, we conclude that
\begin{align*}
    \mathcal{D}_{x_\star,y_\star}(\bar{x}_K,\bar{y}_K)\leq\frac{1}{K}\sum_{k=0}^{K-1}\mathcal{D}(x_{k+1},y_{k+1})=\frac{1}{K}\sum_{k=0}^{K-1}\p{F_k+G_k}\leq \frac{V_0}{K},
\end{align*}
where $\mathcal{D}_{x_\star,y_\star}$ is defined in \cref{eq:pd_gap_fcn} and the equality follows from \cref{lem:FkGk_properties}. 
\end{proof}

\subsection{Proof of \texorpdfstring{\cref{thm:sequence}}{Theorem 2}}

\begin{proof}
Due to \cref{lem:lyapunov-inequality,lem:coeff_pos,lem:FkGk_properties,lem:coeff_nonneg}, we conclude that $\p{V_k}_{k=0}^{\infty}$ converges. This 
in turn implies that $\p{x_k}_{k=0}^{\infty}$ and $\p{y_k}_{k=0}^{\infty}$ remain bounded. Next, the same lemmas also imply, using a telescoping summation, that $F_k\to 0$,
\begin{align}
\frac{\p{4 \theta^2 - 1} \p{4 - \tau  \sigma  \norm{L}^2 \p{2 \theta + 1}} \p{4 - \tau  \sigma  \norm{L}^2 \p{2 \theta - 1}}}{16 \tau  \p{8 \theta - \tau  \sigma  \norm{L}^2 \p{4 \theta^2 + 1}}} \norm{x_{k+1} - x_k}^2&\to 0,\label{eq:pf:xk_resid_conv}
\end{align}
and
\begin{align}
\frac{1}{2 \sigma } \norm{y_{k+1} - y_k - \sigma  \p{\frac{1}{2} \p{L x_{k+1} - L x_{k+2}} - \theta \p{L x_k - L x_{k+1}}}}^2 &\to 0\label{eq:pf:yk_resid_conv}
\end{align}
as $k\to\infty$. Since the coefficients are nonzero due to the parameter assumptions and \cref{lem:coeff_nonneg}, we conclude from \cref{eq:pf:xk_resid_conv} that
\begin{align}\label{eq:x_residual_to_zero}
    \|x_{k+1}-x_k\|\to 0   
\end{align}
as $k\to\infty$. Combining \cref{eq:pf:yk_resid_conv} and \cref{eq:x_residual_to_zero} gives that 
\begin{align}\label{eq:y_residual_to_zero}
    \|y_{k+1}-y_k\|\to 0
\end{align}
as $k\to\infty$, since $L$ is a bounded linear operator. Moreover, \cref{eq:x_residual_to_zero} combined with the fact $\p{V_k}_{k=0}^{\infty}$ converges and $F_k\to 0$ as $k\to\infty$ give that 
\begin{align*}
    \p{\frac{1}{2 \tau } \norm{x_{k+1} - x_\star}^2 + \frac{1}{2 \sigma} \norm{y_k - y_\star + \sigma  \theta \p{L x_{k+1} - L x_k}}^2 }_{k=0}^{\infty}
\end{align*}
converges. Next, since $\p{y_k}_{k=0}^{\infty}$ is bounded, \cref{eq:x_residual_to_zero,eq:y_residual_to_zero} imply that 
\begin{align*}
    &\|y_{k+1}-y_\star\|^2 - \norm{y_k - y_\star + \sigma  \theta \p{L x_{k+1} - L x_k}}^2  \\
    &\quad = \|y_{k+1} - y_{k}\|^2 -  2\inner{y_k - y_\star}{\sigma  \theta \p{L x_{k+1} - L x_k} - \p{y_{k+1} - y_{k}} } \\
    &\quad\quad\quad - \norm{\sigma  \theta \p{L x_{k+1} - L x_k}}^2 \to 0
\end{align*}
as $k\to \infty$, from which we conclude that 
\begin{align*}
    \p{\frac{1}{2 \tau } \norm{x_{k+1} - x_\star}^2 + \frac{1}{2 \sigma} \|y_{k+1}-y_\star\|^2  }_{k=0}^{\infty}
\end{align*}
converges.

Construct the sequence $\p{z_k}_{k=0}^{\infty}$ such that $z_k=(x_{k+1},y_{k+1})$ for each nonnegative integer $k$ and let $Z=\{(x_\star,y_\star)\in\PrimS\times\DualS : \eqref{eq:kkt} {\hbox{ holds}}\}$ be the KKT solution set. Since $\p{z_k}_{k=0}^{\infty}$ is bounded, it has weakly convergent subsequences. Let $\p{z_{k_n}}_{n=0}^{\infty}$ be such a subsequence, with weak limit point $\bar{z}=(\bar{x},\bar{y})\in\PrimS\times\DualS$ say. Then the inclusion formulation of the algorithm in \cref{eq:CP-inclusion:f,eq:CP-inclusion:g} satisfies
\begin{align*}
   \partial f \p{x_{k_n+1}}+L^* y_{k_n}\ni\frac{1}{\tau } \p{x_{k_n} - x_{k_n+1}} \to 0,  \\
   \partial g^* \p{y_{k_n+1}} -Lx_{k_n+1}  \ni \frac{1}{\sigma } \p{y_{k_n} - y_{k_n+1}} + \theta L  \p{x_{k_n+1} - x_{k_n}} \to 0,
\end{align*}
as $n\to\infty$, since $x_{k_n+1}-x_{k_n}\to 0$ and $y_{k_n+1}-y_{k_n}\to 0$ as $n\to\infty$. Let us introduce the operator $A:\PrimS\times\DualS\rightrightarrows \PrimS\times\DualS$ defined as
\begin{align*}
    A(x,y)= \partial f(x)\times \partial g^*(y) +(L^*y,-Lx)
\end{align*}
for each $(x,y)\in\PrimS\times\DualS$.
This is a maximally monotone operator since it is the sum of two maximally monotone operators, the first one due to \cite[Theorem~20.25,Proposition~20.23]{Bauschke_Combettes_2017} and the second one due to \cite[Example~20.35]{Bauschke_Combettes_2017}, one of which has full domain \cite[Corollary~25.5]{Bauschke_Combettes_2017}. Therefore, 
\begin{align*}
A(x_{k_n+1},y_{k_n+1})\ni\p{\frac{1}{\tau } \p{x_{k_n} - x_{k_n+1}} - L^*\p{y_{k_n} - y_{k_n+1}}, \frac{1}{\sigma } \p{y_{k_n} - y_{k_n+1}} + \theta L  \p{x_{k_n+1} - x_{k_n}}} \to 0
\end{align*}
and we conclude, using weak-strong closedness of maximally monotone operators \cite[Proposition~20.38]{Bauschke_Combettes_2017}, that $0\in A(\bar{x},\bar{y})$, implying
\begin{align*}
  0\in \partial f \p{\bar{x}}+L^* \bar{y}, \\
  0\in\partial g^*(\bar{y})-L \bar{x},
\end{align*}
i.e., the weak limit point is a KKT point. Since $\p{\frac{1}{2\tau }\|x_{k+1}-x_\star\|^2+\tfrac{1}{2\sigma }\|y_{k+1}-y_\star\|^2}_{k=0}^{\infty}$ converges and all weak limit points belong to the KKT solution set $Z$, we invoke \cite[Lemma~2.47]{Bauschke_Combettes_2017} to conclude weak convergence to a KKT point.
\end{proof}

\section{Counterexample}\label{sec:counterexample}
This section proves that the result in \cref{thm:sequence} is tight, i.e., there is an example where (weak) convergence fails for each $\theta>1/2$ whenever $\tau\sigma\|L\|^2\geq 4/\p{1+2\theta}$. Let \(\PrimS = \Real\) and \(\DualS = \Real\), and consider problem \eqref{eq:pd_problem} for the case $f = g^{*} = 0$, and $L=1$, i.e.,
\begin{align}\label{eq:counterexample}
    \minimize_{x\in\Real}\maximize_{y\in\Real}\; xy.
\end{align}
The update rule, \eqref{eq:CP-iteration}, for the Chambolle--Pock method then becomes 
\begin{align*}
    x_{k+1} &= x_k - \tau y_k, \\
    y_{k+1} &= y_k + \sigma \p{x_{k+1} + \theta \p{x_{k+1} - x_k}},
\end{align*}
or equivalently 
\begin{align*}
    \begin{bmatrix}
        x_{k+1} \\
        y_{k+1}
    \end{bmatrix} =
    \begin{bmatrix}
        1 & - \tau \\
        \sigma & 1-\tau\sigma\p{1+\theta} 
    \end{bmatrix} 
    \begin{bmatrix}
        x_{k} \\
        y_{k}
    \end{bmatrix},
\end{align*}
where we assume that $\tau,\sigma>0$. Therefore, $\p{x_{k},y_{k}}_{k=0}^{\infty}$ converges for an arbitrary initial point $\p{x_0,y_{0}}\in\Real^2$ to the solution $\p{x_{\star},y_{\star}}=\p{0,0}\in\Real^2$ if and only if both eigenvalues of the matrix above have magnitude less than $1$. The two eigenvalues of the matrix are
\begin{align*}
        \lambda_{1} = \frac{1}{2}\p{2 - \tau\sigma\p{1+\theta} + \sqrt{\tau\sigma\p{\tau\sigma\p{1+\theta}^2 - 4} }}
    \end{align*}
    and
    \begin{align*}
        \lambda_{2} = \frac{1}{2}\p{2 - \tau\sigma\p{1+\theta} - \sqrt{\tau\sigma\p{\tau\sigma\p{1+\theta}^2 - 4} }}.
    \end{align*}
Suppose that the step-size condition of \cref{thm:sequence} is violated, i.e., that $\tau\sigma \geq 4/\p{1+2\theta}$. In this case, the eigenvalues are distinct real numbers satisfying
\begin{align*}
    \lambda_{2} \leq -1 \quad \text{ and } \quad \lambda_{2} < \lambda_{1} < \frac{1}{2}\p{2 - \tau\sigma\p{1+\theta} + \sqrt{\tau^2\sigma^2\p{1+\theta}^2} } = 1.
\end{align*}
In particular, the method fails to converge, which proves that the result in \cref{thm:sequence} is tight.

Note that the boundary condition $\tau \sigma = 4/\p{1+2\theta}$ implies that $\lambda_{2}=-1$. Although the sequence $\p{x_{k},y_{k}}_{k=0}^{\infty}$ does not necessarily converge in this setting, we have that $(x_{k+2},y_{k+2})-(x_k,y_k)\to 0$ as $k\to\infty$, i.e., the difference between every second iterate converges to $0$. It turns out that this is not a coincidence as it always holds for the Chambolle--Pock algorithm on the boundary, as shown next.

\section{Boundary}
\label{sec:boundary}

In this section, we show that whenever  $\theta>1/2$ and \(\tau  \sigma  \norm{L}^2 = {4}/{\p{1 + 2 \theta}}\), the sequence $\p{\|x_{k+2}-x_k\|^2}_{k=0}^{\infty}$ is summable, implying that the difference between every second iterate converges to 0.

Due to \cref{lem:coeff_nonneg}, we conclude that the coefficient in front of $\|x_{k+1}-x_k\|^2$ in the Lyapunov inequality in \cref{lem:lyapunov-inequality} is 0. Moreover,
\begin{align*}
\frac{8 \theta - \tau  \sigma  \norm{L}^2 \p{4 \theta^2 + 1}}{16 \tau } =  \frac{8 \theta(1+2\theta) - 4 \p{4 \theta^2 + 1}}{16 \tau (1+2\theta)}=\frac{2 \theta - 1 }{4 \tau (1+2\theta)}
\end{align*}
and
\begin{align*}
    \frac{4 \p{1 - \tau \sigma  \theta \norm{L}^2}}{8 \theta - \tau  \sigma  \norm{L}^2 \p{4 \theta^2 + 1}}&=\frac{4 \p{1-2\theta}}{\p{8 \theta - \tau  \sigma  \norm{L}^2 \p{4 \theta^2 + 1}}(1+2\theta)}\\
    &=\frac{4 \p{1-2\theta}(1+2\theta)}{(8\theta-4)(1+2\theta)}\\
    &=-1,
\end{align*}
implying that the Lyapunov inequality in \cref{lem:lyapunov-inequality} reads
  \begin{multline*}
    V_{k+1} - V_k + F_k + G_k \\
    \begin{aligned}
      &\leq  - \frac{1}{2 \sigma } \norm{y_{k+1} - y_k - \sigma  \p{\frac{1}{2} \p{L x_{k+1} - L x_{k+2}} - \theta \p{L x_k - L x_{k+1}}}}^2 \\
      &\qquad - \frac{2 \theta - 1}{4 \tau  \p{1 + 2 \theta}} \norm{x_{k+2} - x_k}^2,
    \end{aligned}
  \end{multline*}
  where $V_k$, $F_k$, and $G_k$ are nonnegative for all nonnegative integers $k$. 
  Since $\theta>{1}/{2}$, we conclude using a telescoping summation argument that $\p{\|x_{k+2}-x_k\|^2}_{k=0}^{\infty}$ indeed is summable.

\section*{Declarations}
\subsection*{Funding}
This work was partially supported by the ELLIIT Strategic Research Area and the Wallenberg AI, Autonomous Systems and Software Program (WASP) funded by the Knut and Alice Wallenberg Foundation. S.~Banert and P.~Giselsson acknowledge support from Vetenskapsr\aa{}det grant VR~2021-05710.

\printbibliography

\end{document}